\documentclass[edeposit,fullpage,12pt]{uiucthesis2009}
\usepackage{amsthm}
\theoremstyle{definition}
\usepackage{amsmath}
\usepackage{amssymb}
\usepackage{epsfig} 
\usepackage{epstopdf}
\usepackage[nottoc]{tocbibind}
\usepackage{tikz}
\usetikzlibrary{matrix}
\usepackage[square,sort,comma,numbers]{natbib}
\usepackage{graphicx} 
\usepackage{fancyhdr}
\usepackage{hyperref}   
\usepackage{extramarks}
\usepackage{enumitem}
\usepackage{verbatim}
\usepackage{comment}
\usepackage{mathpazo}
\usepackage{euler}
\usepackage{setspace}
\usepackage[margin=1in]{geometry}
\newtheorem{theorem}{Theorem}[chapter]
\newtheorem{definition}{Definition}[chapter]
\newtheorem{prop}{Proposition}[chapter]
\newtheorem{lemma}{Lemma}[chapter]

\newtheorem{example}{Example}[chapter]

\newcommand{\lb}{\{}

\newcommand{\R}{\mathbb{R}}
\newcommand{\C}{\mathbb{C}}

\newcommand{\N}{\mathbb{N}}

\newcommand{\disp}{\displaystyle\prod}
\newcommand{\disu}{\displaystyle\bigcup}

\newcommand{\diss}{\displaystyle\sum}
\newcommand{\disg}{\displaystyle\int}
\newcommand{\disl}{\displaystyle\lim}
\newcommand{\dislim}{\displaystyle\lim}

\newcommand{\dislimsup}{\displaystyle\limsup}

\newcommand{\dissup}{\displaystyle\sup}
\newcommand{\disinf}{\displaystyle\inf}
\newcommand{\dismax}{\displaystyle\max}

\newcommand{\Lc}[0]{\mathcal{L}}

\newcommand{\Pc}[0]{\mathcal{P}}

\newcommand{\Tc}[0]{\mathcal{T}}

\newcommand{\kf}[0]{\mathfrak{k}}

\newcommand{\tf}[0]{\mathfrak{t}}

\newcommand{\dx}[0]{\dot{x}}

\renewcommand{\a}{\alpha}
\renewcommand{\b}{\beta}
\renewcommand{\d}{\delta}
\newcommand{\e}{\varepsilon}

\renewcommand{\k}{\kappa}

\newcommand{\m}{\mu}

\newcommand{\p}{\pi}
\newcommand{\ph}{\varphi}

\newcommand{\q}{\xi}

\newcommand{\s}{\sigma}
\renewcommand{\t}{\tau}

\newcommand{\x}{\chi}

\newcommand{\ha}[0]{\hat{a}}

\newcommand{\he}[0]{\hat{e}}

\newcommand{\hs}[0]{\hat{s}}

\newcommand{\hx}[0]{\hat{x}}

\begin{document}
\title{Topological Entropy Bounds for Switched Linear Systems with Lie Structure} 
\author{A. James Schmidt}
\department{Electrical and Computer Engineering} 
\schools{B.A., University of Notre Dame, 2014\\
		B.S., University of Notre Dame, 2014\\
		M.S., University of Illinois at Urbana-Champaign, 2016}
		\msthesis
		\adviser{Assistant Professor Mohamed-Ali Belabbas\\
		Professor Daniel Liberzon}
		\degreeyear{2016}
		\committee{Assistant Professor Mohamed-Ali Belabbas\\ Professor Daniel Liberzon}
		\maketitle
\frontmatter

\singlespacing
\begin{abstract}
	In this thesis, we provide an initial investigation into bounds for topological entropy of switched linear systems.  Entropy measures, roughly, the information needed to describe the behavior of a system with finite precision on finite time horizons, in the limit.  After working out entropy computations in detail for the scalar switched case, we review the linear time-invariant nonscalar case, and extend to the nonscalar switched case.  We assume some commutation relations among the matrices of the switched system, namely solvability, define an ``upper average time of activation'' quantity and use it to provide an upper bound on the entropy of the switched system in terms of the eigenvalues of each subsystem. 
\end{abstract}

\tableofcontents
\mainmatter

\renewcommand{\kf}{\mathfrak{h}}
\renewcommand{\k}{\eta}
\chapter{Introduction}

Since its introduction by Kolmogorov in dynamical systems, entropy has been an invaluable tool for understanding  system behavior.  Ornstein's isomorphism theorem, which entirely characterizes---up to isomorophism---a class of systems (Bernoulli) according to their entropy, further solidified its importance in the dynamical systems community \citep{ornstein}, \citep{katokentropy}.  Broadly, entropy captures uncertainty growth as a system evolves.  Entropy itself has a history which precedes its inclusion in dynamical systems---notably in thermodynamics, and much later information theory as well---and this intuition about uncertainty growth corresponds, at least in the development of the concept, to entropy in other disciplines.

Topological entropy generally measures the exponential growth in information needed to approximate trajectories on a finite time scale---equivalently, to cover the flow with the initial data of an open cover---or the growth in complexity of a system acting on finite measurable sets.  The latter idea corresponds to Kolmogorov's original definition \citep{halmos}  and shares a striking resemblance to the entropy of information theory which Shannon introduced roughly ten years earlier \citep{shannon}.  Adler's introduction of entropy---the entropy approximating trajectories---appeared first as a mild extension of Kolmogorov's entropy, quantifying a map's expansion according to  how it increases (by joins) a minimal number of sets needed to cover a compact space \citep{adler}.  Bowen, on the other hand, considered topological entropy by quantifying the increase in number of points whose flow can be separated by a small distance at some time in their initial segments \citep{bowen}.  These definitions are all interrelated, and the arguments laying out their connections can be found unified in \citep{katoksystems}.  Most results in entropy are for autonomous systems, as time dependence in the dynamics introduces complexities which require new methods to understand  \citep{kawantime}, \citep{kolyada}.  Our present work on switched systems is an initial investigation into some of those complexities. 

Entropy has played a useful role in control theory, where information flow appears in maintaining or inducing key properties through actuators and a network of  sensors.  Original applications of entropy tried quantifying the minimal data rates for guaranteeing the possibility of stabilizability, observability, or controllability \citep{tatikonda}, \citep{nair-entropy-tac-ncs}.  These  approaches extended original entropy considerations, notably in allowing for the evolution of a system in a noncompact space, but still described to the complexity of the system (as it evolves) in the state space.  Later in \citep{coloniusinv}, entropy was used to quantify the growth in number of open loop controls needed to span (in another sense) the space of controls.  This work, and similarly \citep{coloniusnote} and \citep{coloniuscontrol}, used entropy in the setting of continuous time, whereas most previous investigations focused on discrete time systems. Also, the application of topological entropy in the dual problem of state estimation has been studied in, e.g.,\  \citep{savkin-entropy-automatica} and the recent work~\citep{LM-HSCC-2016}.

Of interest to us in the class of dynamical systems are \textit{switched systems}, those which are ``piecewise time invariant,'' formally given by a family of systems with a switching signal---usually dependent on time---which determines which system is active \citep{liberzonswitch}.  Switched systems provide a very nice next-step extension beyond time-invariant systems and insight on how to broach the \textit{prima facie} intractable abstractness of general nonlinear time-varying systems.  Notably, nice properties which individual subsystems may enjoy---such as stability, or equivalently all negative real part eigenvalues---may not hold for the entire switched system. Stability of switched systems has been extensively studied, and among the various techniques for understanding stability are the use of Lyapunov functions, constraints on the frequency of switching, and \textit{commutation relations} among the subsystems.   

We explain this last approach. Suppose we are given a \textit{linear} switched system whose dynamics are comprised of finitely many matrices. We can consider the Lie algebra generated by these matrices---i.e.,\ the set of matrices generated by taking brackets of a finite set of matrices, brackets of those, and so on---and ask whether any structure can be used to determine conditions for stability of the switched system.  Indeed, it turns out that if the Lie algebra is solvable (which means iterated commutators of a certain form eventually vanish) and each individual system is stable, then the switched system is stable \citep{kutepov-82,lie}.  In other words, stability is invariant (or ``stable'') under any switching signal provided stability of the underlying systems. In concrete terms, solvability corresponds to the existence of a change of basis transformation which \textit{simultaneously} takes each matrix to an upper triangular one \citep{serrelie}.  Solvability generalizes commutativity of matrices, and the latter can be easily seen to directly imply the preservation of stability:  by stacking the solution for each system together,  there is a uniform bound on the number of switches \textit{in the solution}.   Naturally things become more complicated for nonlinear systems, but results tying nonlinear systems to the Lie algebra structure can be found in \citep{liberzonmar} and \citep{sharon-margaliot}.  Research into robustness conditions can be found in  \citep{ABL-cdc10}, and the works of \citep{haimovich-braslavsky-tac-13,haimovich-braslavsky-cdc10}  investigate state feedback which induces simultaneous triangularizability in the closed loop.  Other references to switched systems can be found in \citep{liberzonswitch} and \citep{shorten-wirth-...-survey}.
 
We are interested in understanding entropy in switched systems.  Our motivations are twofold. First of all, to our knowledge, little is understood about the problem of computing entropy in switched systems. Secondly, we expect that entropy of switched systems may elucidate how to control switched systems  operating over finite data-rate communication channels. Though the linear time-invariant (LTI) case is completely understood---there is a tight relation between entropy and the minimal data rate required for ``good'' properties---not much is known about corresponding results in the case of switched systems.  Work in this direction has been started in \citep{switched-quantized-cdc11} and \citep{oliver-acc15} and our work with entropy aims at contributing to and developing these efforts.

Our work is confined to \textit{linear} switched systems.  There are a few reasons for this.  First, entropy---unlike stability---is not a local property of solutions.  Entropy measures, more or less, the exponential expansion of a system, and exponential behavior is through and through a feature of linear systems. Secondly, linear systems provide an amenable context for the use of commutation relations, which indeed grounds the bulk of our machinery. 

 Our results include the following.  First, we  investigate entropy properties of scalar switched systems, the computation for which is the basis of our entropy considerations in higher dimensions.  
 	For each subsystem, we define a function of time which represents the percentage time of activation on that time interval, and use the limit supremum as the upper average time percent of activation.  We use these values as weights on the eigenvalues of their respective systems, and take the sum to represent the upper average exponent of the solution.  We show that this quantity is in fact the entropy of the scalar switched system. 
 	
 We then work out a derivation of the entropy expression for continuous-time LTI systems which, though well known, is not well documented in the literature. We fuse this into a study of the nonscalar switched case, which forms our main results. We provide a general lower bound on entropy independent of structure (and therefore weak).    We then consider bounds on entropy using eigenvalue analysis given certain Lie structure, including commutativity and solvability.  The bound is given by the following: assuming solvability of the Lie algebra which the matrices of the switched system generate, there exists a single change of basis matrix making all of them upper triangular.  Each row represents a scalar linear system, whose state is given by the diagonal entry and ``inputs'' the off diagonal entries.  We show that the effects of expansion in each direction are maximized in the first row, with the following relation: the $j$-th state solution---as input---expands at the exponential rate given by the sum of entropies of scalar states starting at the first state and ending at itself.  Thus the first state's entropy appears $n$ times, the second appears $n-1$ times, and so on, so that the entropy of the whole system is bounded above by the sum of the scalar entropies of the $j$-th subsystem (without inputs, or equivalently weighted sum of diagonal entries) multiplied by $n+1-j$, as $j$ runs up to the dimension of the system.  Note that in order to compute this upper bound, the system must be first brought into the normalized (simultaneous triangularizable) form, as each  state's entropy depends on which eigenvalues are associated with it.

The structure of the thesis is as follows.  In \S 2, we review basic definitions of topological entropy and switched systems. In \S 3 we provide an initial step towards computing entropy of switched systems.  In \S 4, we recount entropy for LTI systems, and provide a proof in continuous time. The final chapter, \S5, includes the main results on nonscalar switched systems.  This includes a lower bound given by a weighted sum of the eigenvalues from each subsystem.  We provide bounds---both lower and upper---in the case where the system can be simultaneously diagonalized, and end with an upper bound on a system which is simultaneously triangularizable.

\chapter{Entropy Preliminaries} 
\indent\indent 
Consider a time-varying dynamical system  \begin{equation}\label{eq:system} \dx=f(t,x)\end{equation}  evolving as $x(t)\in \R^n$, with $f(x)\in T_x\R^n\cong \R^n$,  and fix a compact set $K\subset\R^n$ of initial conditions containing the origin, which we assume to be an equilibrium point $f(0)=0$.  Let $\varphi_f(t,x_0)$ denote solution of system \eqref{eq:system}  at time $t$ starting from initial condition $x(0)=x_0$. When the system is fixed, we will drop the dependence of $\ph$ on $f$, or if the system depends on some parameter---as a switched system will depend on the switching signal $\s$---we may only include the parameter to disambiguate.   For $T\geq 0$, we let $\ph(T,K):=\{\ph(T,x):\,x\in K\}$ the set of points where the solution ends up after $T$-units of time starting from a point in $K$.

We are interested in quantifying how much information is needed in order to approximate solutions starting in $K$.  Fix $T,\e>0$ and consider the flow starting initial conditions from $K$ on some finite time horizon $[0,T]$.  We say that a discrete set $S=\{x_1,\ldots,x_k\}\subset K$ is \textit{$(T,\e)$-spanning} if for each $x_0\in K$ there is an $x_i\in S$ satisfying \begin{equation} ||\varphi(t,x_0)-\varphi(t,x_i)||<\e\end{equation} for every $t\in [0,T]$. Which norm $||\cdot||$ we use on $\R^n$ doesn't particularly matter, but for concreteness we will generally take the infinity norm defined by \begin{equation} ||(x_1,\ldots,x_n)||_\infty:=\max_{j=1,\ldots,n}|x_n|.\end{equation}

The condition that $||\varphi(t,x)-\varphi(t,y)||<\e$ for all $t\in[0,T]$ is equivalent to  $$\dissup_{t\in[0,T]}||\varphi(t,x)-\varphi(t,y)||<\e,$$ on account of which inequality we define a norm-induced metric on  function spaces as follows: for  $f,g:\R\rightarrow\R^n$ $$||f(\cdot)-g(\cdot)||_{[0,T]}:=\dissup_{t\in[0,T]}||f(t)-g(t)||_\infty.$$  Though we will always include the time interval the function norm is ranging over, we will sometimes write $||\cdot||_{\Lc_\infty}$ for $||\cdot||_{[0,\infty)}:=\disl_{T\rightarrow\infty}||\cdot||_{[0,T]}$.

Let $s(T,\e)$ denote the minimal cardinality of a $(T,\e)$-spanning set and define the entropy of system \eqref{eq:system} to be \begin{equation} h(f):=\disl_{\e\rightarrow0}\limsup_{T\rightarrow\infty}\frac{\log(s(T,\e))}{T},\end{equation} where we will by default use the natural logarithm.  This base is more amenable to our purposes in that our setting is in continuous time, so the choice will eliminate  the need for an extraneous multiplicative factor in our results, but translating these notions for the sake of encoding (say, with binary strings) may require using base $2$. 

In the definition above, we fixed some compact set $K$ but were rather cavalier about which one to fix.  For simplicity in the computations to follow, we take $K$ to be a closed ball (cube) in $\R^n$ of integer length.  As will become evident in those computations, the size of $K$ really does not matter, but what may not be as obvious is that the shape does not either.  It is a technical construction to verify this, and we omit the proof for simplicity as it elucidates little in terms of entropy.  Suffice it to say that for ``oddly shaped'' sets (e.g.\ cantor-like nowhere dense closed ones), we can embed them into nice ones, and guarantee a uniform bound, dependent only on dimension, for the multiplicative factor on the cardinality of spanning sets.  This factor disappears asymptotically, and we will see why in the computations. 

There is a second definition of entropy which we will use briefly in our initial investigation. Take, again, a discrete set $S=\{x_1,\ldots,x_k\}\subset K$; this time we say that $S$ is $(T,\e)$-\textit{separated} if for every $x_i\neq x_j\in S$, $||\ph(t,x_i)-\ph(t,x_j)||_{[0,T]}\geq \e$, which amounts to the same as: the trajectories starting from $x_i$ and $x_j$ are at least $\e$ far away at some time $t\in[0,T]$. We let $n(T,\e)$ denote the \textit{maximal} cardinality of a $(T,\e)$-separated set and again define entropy as $$h(f):=\disl_{\e\rightarrow0}\limsup_{T\rightarrow\infty}\frac{\log(n(T,\e))}{T}.$$  One can quickly check that these definitions agree; a maximal $(T,\e)$-separated set is also $(T,\e)$-spanning so $n(T,\e)\geq s(T,\e)$, and an $\e$-ball does  not contain points separated by distance greater than $2\e$, so $s(T,\e)\geq n(T,2\e)$, c.f.\ \citep{katoksystems},\citep{liberzonentropy}.

We are interested in entropy of a restricted class of linear time-varying systems, namely switched systems $$\dx=f(t,x)=f_{\s(t)}(x),$$ where $\sigma:\R\rightarrow \mathcal{P}$ is a switching signal; there is no a priori restriction on the cardinality of $\Pc$, but usually it will be at least compact and in most of our examples even finite.  The notation means the following: for $\Pc:=\{t\in\R:\s(t^+)\neq s(t^-)\}$, which we enumerate canonically as $\{t_i\}_{i\in N\subset \N}$ (where $N=\{1,2,\ldots,n_f\}$ if $N\neq \N$) and $t_0=0$, the system $\dx=f_\s(x)$ is piecewise time-invariant on each interval $(t_i,t_{i+1})$.   As switching does not in general preserve stability for a system whose individual dynamics are stable, we will use entropy to measure the emergence of instability in a switched system.

  To investigate properties of switched entropy, we will fix a class $\Sigma$ of switching signals corresponding to a class of dynamical systems $\dx = f_\s(x)$ with $\s\in \Sigma$.  We will then fix the switching signal itself and determine bounds on the entropy.  As this procedure will use nothing particular about our choice of $\sigma$, we will conclude that the bound holds for each switched system in the class. (Notice the similarity with $\d$-$\e$ arguments in analysis; this is generally how all formal $\forall$-arguments go \citep[\S30]{quine}.)

In order to situate the approaches which follow, we make an initial observation.  At a high level, the expansion which entropy is capturing corresponds to the eigenvalues of a (linear) system. On the other hand, the reason that switching can destabilize otherwise stable systems is that it mixes the evolution of solutions from different eigen\textit{spaces}. Because the geometry of mixing is simply intractable, our approach is to find conditions on the shared structure of individual dynamics which allow us to elicit information about the entropy from the eigenvalues of each system. 

\chapter{Scalar Switched Systems}
\section{Stability}

Start with a scalar system \begin{equation}\label{eq:switch} \dx=a_\sigma x\end{equation} and discrete switching signal $\sigma:[0,\infty)\rightarrow \{1,\ldots,k\}$ and let $\chi_i$ represent the indicator function on system $i$: $$ \x_i(s):= \left\{\begin{array}{lll} 1 & \mbox{if}& \s(s)=i\\ 0 & \mbox{else} & \end{array}\right..$$  We define $$\t_i(t):=\disg_0^t\x_i(s)ds$$ as the total time mode $i$ is active on $[0,t]$ and define its ``average in the limit'' as $$\tf_i:=\dislimsup_{t\rightarrow\infty}\dfrac{1}{t}\t_i(t).$$ Note that the limit may not exist, in which case it is not precisely an average, but our use of the quantity in entropy computations does not require the limit to exist.  In general,   $\diss_{i=1}^k\tf_i\leq k$ and when the limit $\tf_i=\dislim_{t\rightarrow\infty}\dfrac{1}{t}\t_i(t)$ exists for each $i$, it is easy to see that $\diss_{i=1}^k\tf_i = 1$.  

Before proceeding, we introduce additional notation, which will be ubiquitous in what follows.  Let $$\begin{array}{lcl}   \k(t)&:=&\diss_{i=1}^ka_i\t_i(t), \\ \kf & := & \diss_{i=1}^k a_i\tf_i=\dislimsup_{t\rightarrow\infty}\frac{1}{t}\k(t)\\ \kf^+& := &\max\{0,\kf\}.\end{array}$$ 
As we will see in proposition \ref{prop:scalin}, it is not necessary in the computation of entropy for the limit of $\frac{1}{t}\t_i(t)$ to exist. We will still, for theoretical interest, investigate conditions under which it does. We will, irrespective of existence of the limit, refer to $\kf$ as the average asymptotic exponent, or just average exponent, and $\tf_i$ as the average asymptotic time of activation for mode $i$. 

 Notice that these definitions depend on the switching signal, but as we generally fix it from the outset, the dependence will remain notationally implicit. 

First we observe that when a scalar switched  system has average negative exponent, it is stable. 
\begin{prop}\label{prop:scalarstable} Suppose for a switched system \eqref{eq:switch} that the average exponent $$ \kf= \diss_{i=1}^ka_i\tf_i<0$$ is negative, with $\kf$ and $\tf_i$ defined as above.   Then the system is globally exponentially stable. \end{prop}  
\begin{proof}Because scalar exponentials commute, the solution to \eqref{eq:switch} is given by $$ \varphi(t,x_0)=e^{\k(t)}x_0 $$ and as $\dislimsup_{t\rightarrow\infty}\frac{1}{t}\k(t)=\kf$, for each positive $\hat{\e}<-\kf$, there is $T_{\hat{\e}}$ such that $$ \frac{1}{t}\k(t)<\kf+\hat{\e}$$ whenever $t>T_{\hat{\e}}$. For such $t$, $$\k(t)<(\kf+\hat{\e})t,$$ and we define $\kf+\hat{\e}=:-\lambda<0$. Then $$ |\varphi(t,x_0)|=e^{\k(t)}|x_0|\leq e^{-\lambda t}|x_0|,$$ when $t>T_{\hat{\e}}$, which proves convergence. 
	
	For stability, let $\e>0$ be given.  Then $|\varphi(t,x_0)|\leq e^{-\lambda t}|x_0|$ whenever $t>T_\e$ for some $T_e>0$, with $\lambda$ positive, by the preceding argument on convergence. On the other hand, set $a_{\max}:=\max\{a_1,\ldots,a_k,0\}$, so  $$ |\varphi(t,x_0)|\leq e^{a_{\max} t}|x_0|\leq e^{a_{\max} T_\e}|x_0|$$ for all  $t\leq T_\e$, by the comparison principle.  
	
	Because $a_{\max}\geq 0>-\lambda$, we have $$ |\ph(t,x_0)|\leq e^{a_{\max}T_\e}|x_0|$$ for all time $t\geq 0$.
	
	  Thus $|\ph(t,x_0)|<\e$ as long as $|x_0|<\e e^{-a_{\max}T_\e},$ proving stability. \end{proof}

\section{Entropy}

 \begin{prop}\label{prop:scalin} Fix switching signal $\s:[0,\infty)\rightarrow\{1,\ldots,k\}$ and consider the  scalar  switched linear system as in Equation \eqref{eq:switch} with each system having average time of activation $\tf_i$ for $i=1,\ldots,k$ and average exponent $\kf$ as defined above.   The entropy of the system is: \begin{equation} h(a_\sigma)=\kf^+.\end{equation}  \end{prop}

\begin{proof}  Suppose, first,  that $\kf>0$ and let $\delta>0$.  We will show that $h(a_\sigma)\leq \kf+\delta$.  Let $K=[\a,\b]\subset\R$ be a closed interval which without loss of generality we take to have even integer length $\b-\a=:\ell\in 2\N$, and fix $T,\e>0$.  Let $x_0,\hat{x}_0\in K$.  Then $$ |\varphi(t,x_0)-\varphi(t,\hat{x}_0)|=|e^{\k(t)}(x_0-\hat{x}_0)|=e^{\k(t)}|x_0-\hat{x}_0|.$$ Because $\dislimsup_{t\rightarrow\infty}\frac{1}{t}\k(t)= \kf$, there is $T_\delta>0$ such that \begin{equation}\label{limdelta} \frac{1}{t}\k(t)<\kf+\delta\end{equation} whenever $t>T_\delta$. Then  $\k(t)<(\kf+\delta)t$ for all $t>T_\delta$.  
	
	Consequently, $$|\varphi(t,x_0)-\varphi(t,\hat{x}_0)|<e^{(\kf+\delta)t}|x_0-\hat{x}_0|$$ for all $t>T_\delta$, and for $t\leq T_\delta$, with $a_{\max}:=\max\{a_1,\ldots,a_k,0\}$ as defined in the proof of proposition \ref{prop:scalarstable}, $$ ||\varphi(\cdot,x_0)-\varphi(\cdot,\hat{x}_0)||_{[0,T_\d]}=\dismax_{t\in[0,T_\d]} e^{a_{\max}t}|x_0-\hat{x}_0|= e^{a_{\max}T_\delta}|x_0-\hat{x}_0|,$$ which, in particular, is finite. 
	
	Putting these together, $$\begin{array}{ll} ||\ph(\cdot,x_0)-\ph(\cdot,\hx_0)||_{[0,T]} & \leq \max\left\{||\ph(\cdot,x_0)-\ph(\cdot,\hx_0)||_{[0,T_\d]}, ||\ph(\cdot,x_0)-\ph(\cdot,\hx_0)||_{[T_\d,T]}\right\}\\ 
	&  = e^{\ha_T}|x_0-\hx_0|\end{array}$$ with $\ha_T:=\max\{a_{\max}T_\d,(\kf+\d)T\}$.

Define $\hs_T:=\frac{1}{\e}e^{\ha_T}$ and $\hat{\e}_{T}:=\hs_{T}^{-1}$.  Select $S_{T,\e}\subset K$ to be $\ell\hat{s}_T/2$ evenly spaced points separated by distance $2\hat{\e}_T$ (recall $\ell=\b-\a$ is even), $$S_{T,\e}:=\{\alpha+\hat{\e}_T,\alpha+3\hat{\e}_T,\ldots,\b-\hat{\e}_T\},$$ which by construction is $(T,\e)$-spanning.  Then $$ \log(\# S_{T,\e})= \log\left(\frac{\ell}{2\e}(e^{\ha_T})\right) = \log\left(\frac{\ell}{2\e}\right) +\ha_T, $$ which upper bounds the  cardinality $s(T,\e)$ of a \textit{minimal} $(T,\e)$-spanning set.

 Now we can compute a bound on the entropy: $$\begin{array}{ll} h(a_\s) & := \disl_{\e\rightarrow0}\dislimsup_{T\rightarrow\infty}\frac{1}{T}\log s(T,\e) \\ 
& \leq \disl_{\e\rightarrow0}\dislimsup_{T\rightarrow\infty} \frac{1}{T}\log \#S_{T,\e} \\
 & = \disl_{\e\rightarrow\infty}\dislimsup_{T\rightarrow\infty}\frac{1}{T}( \log\left(\frac{\ell}{2\e}\right) +\ha_T)\\
 & =  \disl_{\e\rightarrow0}\dislimsup_{T\rightarrow\infty}\frac{1}{T}\big(\log\left(\frac{\ell}{2\e}\right)+\max\{a_{\max}T_\d,(k_p+\d)T\}\big)\\ 
 & =\kf+\d,\end{array} $$ where the last equality follows from the fact that neither $\kf$ nor $\d$ depends on $\e$.  Since $\d>0$ is arbitrary, this sequence of relations shows that $h(a_\s)\leq \kf$.

To prove the opposite bound, we use a measure argument.  With $T,\e>0$, suppose that $S_\e(T)$ is any $(T,\e)$-spanning set.  As $\ph(T,K)=\{\ph(T,x):\,x\in K\}$ is covered by $\#S_{T,\e}$ intervals of length $2\e$, the measure of $\ph(T,K)$ must be bounded above by  $2\e\#S_{T,\e}$,  and therefore the number of points needed in order to $(T,\e)$-span is lower bounded by $\frac{1}{2\e}\m(\ph(T,K))$.  We compute: $$\begin{array}{ll} \m(\ph(T,K)) & = \m(e^{\k(T)}K)\\ & = \m(e^{a_k\t_k(T)+\ldots+a_1\t_1(T)}K)\\ & = \m(e^{a_k\t_k(T)}e^{a_{k-1}\t_{k-1}(T)+\ldots+a_1\t_1(T)}K)\\ & =e^{a_k\t_k(T)} \m(e^{a_{k-1}\t_{k-1}(T)+\ldots+a_1\t_1(T)}K)\\& \;\;\;\;\;\;\;\vdots\\& = e^{a_k\t_k(T)}\cdots e^{a_1\t_1(T)}\m(K)\\ & =e^{\k(T)}\m(K).\end{array}$$ The fourth line follows from scale invariance of Lebesgue measure. 

As $\dislimsup_{T\rightarrow\infty}\frac{1}{T}{\k(T)}=\kf$, given any $\d>0$ we have a sequence $\{t_i\}\xrightarrow{i\rightarrow\infty}\infty$ such that $\kf-\d<\frac{1}{t_i}\k(t_i)$ for all $i$.  Thus, we have $$\begin{array}{ll}\dislimsup_{T\rightarrow\infty}\frac{1}{T}\log(\#S_{T,\e}) & >\dislimsup_{T\rightarrow\infty}\frac{1}{T}\big(\log(\ph(T,\e))-\log(\frac{\m(K)}{2\e})\big) \\&= \dislimsup_{T\rightarrow\infty}\frac{1}{T}\big(\k(T)-\log(\frac{\m(K)}{2\e})\big) 
\\ & = \dislimsup_{T\rightarrow\infty}\frac{1}{T}\k(T)
\\ & \geq \dislimsup_{i\rightarrow\infty}\frac{1}{t_i}\k(t_i)
\\ & \geq \kf-\d.\end{array}$$  Because this holds for arbitrary $(T,\e)$-spanning set and for any $\d>0$, it also holds for any minimal $(T,\e)$-spanning set as well and for $\d=0$, proving the lower bound.

 To complete the proof, we consider the case where $\kf\leq 0$ and start by supposing that $\kf=0$., Consider the dynamical system \begin{equation}
\dx = (a_\sigma+\e)x.
\end{equation} From the above, $\kf+\e=\e>0$, so $h(a_\sigma+\e)=\e$.  By the comparison  principle any solution to $\dx = a_\sigma x$ will be bounded by a solution with same initial condition to $\dx= (a_\sigma+ \e)x$.  Because these systems are linear, the same is true for separation of solutions: $$\begin{array}{ll} |\ph_\s(t,x_0)-\ph_\s(t,y_0)| & =|\ph_\s(t,x_0-y_0)|\\ & \leq |\ph_{\s+\e}(t,x_0-y_0)|\\ & = |\ph_{\s+\e}(t,x_0)-\ph_{\s+\e}(t,y_0)|.\end{array}$$	  Hence $h(a_\sigma)\leq h(a_\sigma+\e)$ and as this is true for every $\e>0$, $h(a_\sigma)\leq 0$.  Since $h\geq 0$ by definition, this proves that $h(a_\sigma)=0$. 

For $\kf<0$, we can either apply the preceding argument (this time with $\e>-\kf$, so that $\e+\kf>0$ arbitrary), or use the previous result, Proposition \ref{prop:scalarstable}: for $x_0$ small enough $||\ph(\cdot,x_0)||_{\Lc_\infty}<\e$ and then by linearity for initial points $x_0, y_0$ close  enough, their separation is similarly bounded $||\ph(\cdot,x_0)-\ph(\cdot,y_0)||_{\Lc_\infty}=||\ph(\cdot,x_0-y_0)||_{\Lc_\infty}<\e$. Hence a fixed, finite, number of initial points from $K$ is needed to $(T,\e)$-span for every $T\geq 0$.  \end{proof}

It is possible also to show the lower bound using separated sets.  We provide this argument as it gives another point of view to see that entropy  is well defined even when the limit of $\t_i(t)/t$ does not exist.  The measure argument provided above paves the way for most arguments in higher dimensional cases, but separated sets will reappear in the proof of Proposition \ref{prop:simdiag} .  We will show analogously that $h(a_\sigma)\geq \kf-\delta$ for arbitrary $\delta>0$, and   therefore that $h(a_\sigma)\geq \displaystyle\inf_{\delta>0}(\kf-\delta)=\kf$. 

Since $\dislimsup_{t\rightarrow\infty}$$\frac{1}{t}\k(t)=\kf$, there is an increasing sequence $\Tc=\{t_1,t_2,\ldots\}\subset \R$ with $\disl_{i\rightarrow\infty}t_i=\infty$ for which $$\kf-\d<\frac{1}{t_i}\k(t_i)$$ for every $t_i\in \Tc$.  For such $t_i$, $t_i(\kf-\d)<\k(t_i)$ and consequently $$ e^{(\kf-\d)t_i}|x_0-\hx_0|\leq e^{\k(t_i)}|x_0-\hx_0| \leq ||\ph(\cdot,x_0)-\ph(\cdot,\hx_0)||_{[0,t_i]}.$$

Let $\ha_T:=(\kf-\d)T$, $\hs_T=\frac{1}{\e}e^{\ha_T}$ and $\hat{\e}_T=\hs^{-1}$, and define $S_{T,\e}$ as before, with this new $\hat{\e}_T$, with cardinality $\hs_T\ell/2$; it is $(T,\e)$-separated by construction whenever $T\in\Tc$ since  for such $t_i$, and $x_k,x_j\in S_{T,\e}$, $$||\ph(\cdot,x_k)-\ph(\cdot,x_j)||_{[0,t_i]} \geq ||\ph(\cdot,x_k)-\ph(\cdot,x_k+\hat{\e}_T)||_{[0,t_i]} \geq e^{(\kf-\d)t_i}\hat{\e}_T=\e.$$

Counting, $\log(\#S_{T,\e}) = \log(\ell/2\e)+\ha_T$ which is  bounded above by $n(T,\e)$, the cardinality of a maximal $(T,\e)$-separated set, for $T\in \Tc$.  Now applying the definition of entropy with separated sets, we have $$\begin{array}{ll} h(a_\s) & := \disl_{\e\rightarrow0}\dislimsup_{T\rightarrow \infty}\frac{1}{T}\log n(T,\e)\\ & \geq \disl_{\e\rightarrow\infty}\dislimsup_{T\in \Tc}\frac{1}{T}\log n(T,\e)\\ & \geq \disl_{\e\rightarrow\infty}\dislimsup_{T\in \Tc}\frac{1}{T}\log \#S_{(T,\e)} \\
&= \disl_{\e\rightarrow\infty}\dislimsup_{T\in \Tc}\frac{1}{T}(\log(\ell/2\e)+\ha_T)\\ & 
=\disl_{\e\rightarrow\infty}\dislimsup_{T\in \Tc}\frac{1}{T}(\log(\ell/2\e)+(\kf-\d)T)\\ &= \kf-\d. 
\end{array}$$ Since $\d<0$ is arbitrary, this sequence of relations shows that $h(a_\s)\geq \kf$.

Nothing in our proof required that the switching signal have finite image.  It is possible that $\sigma:[0,\infty)\rightarrow \N$, though we will need  the limit supremum $\kf=\diss_{i<\infty}a_i\tf_i$ to be finite if we want a bound on entropy.  For example, consider the following recursively defined switching signal: on $[0,1)$ (step 0), define $\sigma(t)\equiv a_1$.  At step $i$ (on $[i,i+1)$) define the switching signal by $$ \sigma(t):= \left\lb \begin{array}{lll} \sigma(t-1) & \mbox{if} & \lfloor t\rfloor < 1-\frac{1}{2^{i-1}}\\ a_i & \mbox{else} &\end{array}\right..$$ Clearly each $\tau_i$ has limit $\tf_i=\frac{1}{2^i}$.  For a scalar system with this switching signal, the entropy will indeed be $h(a_\sigma) = \diss_{i<\infty}a_i\tf_i$ as long as the sum is finite.

\subsection{Existence of $\kf$}

\indent\indent Even though for entropy we do not need the limit $\lim\frac{1}{t}\t_i(t)$ to exit, we nevertheless investigate some conditions under which it does. 

\begin{prop}\label{prop:period} Suppose that $\sigma:[0,\infty)\rightarrow\{1,\ldots,k\}$ is periodic (so that for some $T>0$, $\s_p(t)=\s_p(t+T)$ for all $t\in\R_+$). Then $\kf=\disl_{t\rightarrow\infty}\frac{1}{t}\k(t)$ exists. \end{prop} 

Because this statement is a special case of a more general result, we omit the proof.

\begin{definition} By a \textit{folding of length} $T_w$ we mean a partition of $\R_+:=[0,\infty)$ into intervals each having equal length $T_w$: $$[0,\infty)= \disu_{n\in\N}[nT_w,(n+1)T_w).$$ A \textit{window} is an interval of the partition.  \end{definition} Recall that $\t_i(t) =\disg_0^t\x_{\s(s)=i}ds$ gives the total time mode $i$ is active on time interval $[0,t)$.  Then define $\overline{\t}_i(n,T_w):= \t_i((n+1)T_w)-\t_i(nT_w)$, and similarly $\overline{\kf}(n,T_w)=\kf((n+1)T_w)-\kf(nT_w)$. 

\begin{prop}   Suppose there is  a folding of length $T_w>0$ such that $\overline{\t}_i(n,T_w)\neq\overline{\t}_i(m,T_w)$ for only finitely many $n\neq m$, for all $i=1,\ldots,k$.  Then the  limit $ \kf=\disl_{t\rightarrow\infty}\frac{1}{t}\k(t)$ exists. \end{prop} 

This proposition replaces the feature of periodicity responsible for guaranteeing the limit in proposition \ref{prop:period}.  On each period there is an average time each mode is active and that average is the same as the average over $n$ (whole) periods.  But integrating \textit{over} a period, the function introduces some bounded variance, which disappears when we divide by $t\rightarrow\infty$.  

\begin{example} Suppose that $k=2$ and $\tf_i=1/2$. Here is a switching signal which is not periodic but satisfies the conditions of the proposition: let $T_w=1$ and on $[nT_w,(n+1)T_w)$, define $\sigma$ by $$ \sigma\restriction_{[nT_w,(n+1)T_w)}(t):= \left\{\begin{array}{lll} 1 & \mbox{if} & t-T_w\in \frac{[2l,2l+1)}{2^n}\\ 2 & \mbox{else,} & \end{array}\right.$$ where $l=0,\ldots, 2^{n-1}$ partitions the unit interval into even and odd subintervals of length $1/2^n$.  
\end{example}

Still, this condition implies that $\frac{1}{t}\t_i$ has a limit (and therefore $\k/t$ as well), but from $\k_t=\sum a_i\t_i$ it is possible that the left-hand side converges while terms in the sum on the right-hand side do not (e.g. $1=\sin^2+\cos^2$).  Furthermore, we can relax the requirement of equality on all but finitely many windows, as long as the window eventually becomes a good approximation.  
\begin{prop} Suppose there is a folding of length $T_w>0$ for which given any $\e>0$, $|\overline{\k}(n,T_w)-\overline{\k}(m,T_w)|\geq \e$ for only finitely many $n\neq m$ (alt. there is an $N_\e$ such that $|\overline{\k}(n,T_w)-\overline{\k}(m,T_w)|< \e$ whenever $n,m>N_\e$).   Then the  limit $ \k=\disl_{t\rightarrow\infty}\frac{1}{t}\k(t)$ exists.\end{prop} 
\begin{proof} 
Suppose that such a folding of length $T_w>0$ exists and let $T_m(t)= \arg\sup_n\{n\cdot T_w<t\}-1$.   Then computing the integral in window intervals, we have $$ \frac{1}{t}\k(t)=\frac{1}{t}\diss_{n=0}^{T_m}\overline{\k}(n,T_w)+o(t),$$ where $o(t)=\frac{1}{t}\sum a_i(\t_i(t)-\t_i(T_m))\leq k\max_i\{a_1,\ldots,a_k\}T_w$.  As each term is bounded above, it is clear that this term indeed approaches zero as $t\rightarrow\infty$.  Fix $\e>0$; by hypothesis, there is a $\k$ and $N_\e>0$ such that $|\overline{\k}(n,T_w)-\k|<\e$ whenever $n>N_\e$.  Indeed, $\d_\e=\dissup_{n,m>N_\e}|\overline{\k}(n,T_w)-\overline{\k}(m,T_w)|<\e$ implies that $\k(\e)=\inf_{n>N_\e}\overline{\k}(n,T_w)+\frac{\d_\e}{2}$ satisfies this inequality. Since $\e>0$ is arbitrary, we might as well assume that $\k=\liminf_n\overline{\k}(n,T_w)$, and we claim that this is in fact the limit.  

We want to show that there is a $T_\e$ for which $|\frac{1}{t}\k(t)-\k|<\e$ whenever $t>T_\e$.  We compute: $$\frac{1}{t}\diss_{n=0}^{T_m}\overline{\k}(n,T_w)+o(t) = \frac{1}{t}\left(\diss_{n=0}^{N_\e-1}\overline{\k}(n,T_w)+\diss_{n=N_\e}^{T_m(t)}\overline{\k}(n,T_w)\right)+o(t) =\frac{1}{t}\diss_{n=N_\e}^{T_m(t)}\overline{\k}(n,T_w)+o(t).$$ Here $o(t) = \frac{1}{t}\left(\diss_{n=0}^{N_\e-1}(\overline{\k}(n,T_w))+\diss_{i=1}^ka_i(\t_i(t)-\t_i(T_m))\right)$. For $n>N_\e$, we have $|\overline{\k}(n,T_w)-\k|<\e$ and so $$ \frac{1}{t}\k(t)=\frac{1}{t}\diss_{n=N_\e}^{T_m(t)}\overline{\k}(n,T_w)+o(t)\leq \frac{1}{t}\left(\diss_{n=N_\e}^{T_m(t)}\k+\e\right)+o(t)=\frac{1}{t}(T_m(t)-N_\e)(\k+\e)+o(t).$$  Rewriting, we thus obtain: $$ \frac{1}{t}\k(t)\leq \frac{T_m(t)}{t}(\k+\e)+o(t).$$ Since $\frac{T_m(t)}{t}\xrightarrow{t\rightarrow\infty}1$, we obtain $$\dislimsup_{t\rightarrow\infty}\frac{1}{t}\k(t)\leq \dislimsup_{t\rightarrow\infty}\frac{T_m(t)}{t}(\k+\e)+o(t) = \k+\e.$$ Now $\e>0$ is arbitrary, so $\dislimsup_{t\rightarrow\infty}\frac{1}{t}\k(t)\leq \k$.  

For the lower bound, we repeat the same argument using the lower bound $\k-\e<\overline{\k}(n,T_w)$ for $n>N_\e$.  \end{proof}


\chapter{Nonscalar  Linear Time-Invariant Systems}

We restate a well-known result in the theory of topological entropy for linear time-invariant systems, and a proof of the discrete-time case can be found in, e.g.,\ \citep[Theorem 2.4.2]{savkin}, but we here provide an argument in continuous time. 
\begin{theorem}\label{theorem:linear} Let  $$ \dx=Ax$$ be a linear time-invariant system.  Then \begin{equation} h(A) = \diss_{\lambda}\max\{0,\mbox{re}(\lambda)\}=:H(A),\end{equation} where the sum is taken over the eigenvalues of $A$. \end{theorem}

We start with a lemma: \begin{lemma}\label{jordanlemma}
Let $A=\lambda I +N\in M_n(\C)$ be a matrix with $N$ nilpotent $(N^k=0$ some $k>0$),  and $\lambda=\mu+i\nu$.  Then for every $\delta>0$ there is a $T_\delta>0$ such that $$ e^{(\mu+\delta)t}>\displaystyle\sup_{s\in[0,t]}||e^{As}||$$ whenever $t>T_\delta$. \end{lemma} 

\begin{proof} We first show that for $\d>0$ there is a $T_\d'>0$ such that $e^{(\m+\d)t}>||e^{At}||$ when $t>T_\d'$. As $[\lambda I,N]=0$, $$ e^{At}= e^{\lambda I t+Nt}=e^{\lambda I t}e^{Nt}=e^{\lambda It}p(Nt),$$ where $p\in \R[u]$ is some polynomial of degree $n-1$ (namely: $p(u) = 1+u+\ldots+\frac{u^{n-1}}{(n-1)!}$).  Taking norms, we have $$\begin{array}{ll} ||e^{At}|| & = || e^{\lambda It} p(Nt)|| \\
	 & = ||e^{\mu It}e^{i\nu It}p(Nt)|| \\
	  & \leq ||e^{\mu It}||\cdot || e^{i\nu I t}|| \cdot || p(Nt)|| \\
	   & = e^{\mu t} ||p(Nt)||.\end{array}$$ Applying the triangle inequality,  $||p(Nt)||\leq p(||N||t)$, and thus $$ ||e^{At}|| \leq e^{\mu t}p(||N||t),$$ and as for fixed polynomial $p$ and $\d>0$, $e^{\delta t}>p(||N||t)$  whenever $t>T_\d'$ some $T_\d'$ sufficiently large, so the first part follows.  
	   
	   For the statement of the lemma, $||e^{At}||$ is continuous and therefore takes a maximum value on compact set $[0,T_\d']$, call it $r=\dismax_{t\in[0,T_\d']}||e^{At}||$.   Now, $e^{(\m+\d)t}$ is increasing and unbounded on $(0,\infty)$, so there is some $T_\d\geq T_\d'$ for which $e^{(\m+\d)t}>r$ for $t>T_\d$; so when $t>t'>T_\d$, we have $e^{(\m+\d)t}>\max\{e^{(\m+\d)t'},r\} \geq \displaystyle\sup_{s\in[0,t']}||e^{At}||$.\end{proof}

Jordan form matrices canonically satisfy the conditions of this lemma.  The strategy in the following proof is to break up separation of solutions along chunks of Jordan blocks; on each Jordan block, we reduce the problem to the scalar case, and in the end we add up the results. 

\begin{proof}[Proof of Theorem] Let $K\subset \R^n$ be compact containing the origin.   Suppose without loss of generality---after coordinate transformation\citep[\S3.1.b]{katoksystems}---that $A$ is in Jordan canonical form $$ A = \begin{pmatrix} J_1 & \cdots & \textbf{0} \\ \vdots & \ddots& \vdots \\ \textbf{0} & \cdots & J_r\end{pmatrix},$$ with each $J_j\in M_{n_j}(\R)$, $n_j=\dim(J_j)$, for $j=1,\ldots,r$.  Expand $K$ to some cube $\tilde{K}=[\a_1^1,\b_1^1]\times \ldots\times[\a_{n_1}^1,\b_{n_1}^1]\times\ldots\times[\a_1^r,\b_1^r]\times\ldots\times[\a_{n_r}^r,\b_{n_r}^r]\supset K$ with $\b_i^j-\a_i^j=:\ell_i^j\in \N$ integer lengths for simplicity of computation.  

Fix $T, \e>0$; we first show that $h(A)\leq H(A)+n\d$ for any arbitrary $\d>0$, from which we will conclude  that $h(A)\leq H(A)$.  We construct explicitly a $(T,\e)$-spanning set $S_{T,\e}$ which satisfies $\dislimsup_{T\rightarrow\infty}\frac{1}{T}\log\#S_{T,\e}\leq H(A)+\d$.   To this end, we first compute separation of solutions.  For $x, \hx\in \R^n$, $$ \begin{array}{ll} ||\ph(\cdot,x)-\ph(\cdot,\hx)||_{[0,T]} & =\dissup_{t\in[0,T]}||e^{At}(x-\hx)||\\ & 
= \dissup_{t\in[0,T]}\max_{j=1,\ldots,r}||e^{J_jt}(x^j-\hx^j)||\\ &  \leq \dissup_{t\in[0,T]}\max_{j=1,\ldots,r}||e^{J_jt}||||x^j-\hx^j||,
\end{array}$$  where the middle equality holds because we are using the $\infty$-norm and breaking the matrix product along Jordan blocks; the last follows from submultiplicativity. The first equality follows from linearity of the system, and in the last expression $x^j\in \R^{n_j}$ denotes the subvector of $x$ corresponding to submatrix $J_j$ of $A$. 
 
 We now apply Lemma \ref{jordanlemma} with $\m_j$ the real part of Jordan block $J_j$ and some $T_\d>0$, $$\dissup_{s\in[0,t]}||e^{J_js}||<e^{(\m_j+\d)t}$$ whenever $t>T_\d$.  If $\m_j<0$ then $e^{J_jt}$ decays and so  $||e^{J_jt}||$ is bounded.  As we will be taking the limit supremum as $T\rightarrow\infty$, we can assume that that the preceding inequalities hold unqualifiedly.  Then $$ \begin{array}{ll} \dissup_{t\in[0,T]}\max_{j=1,\ldots,r}\max\{e^{(\m_j+\d)t},1\}||x^j-\hx^j||&  \leq \max_{j=1,\ldots,r}\dissup_{t\in[0,T]}\max\{e^{(\m_j+\d)t},1\}||x^j-\hx^j||\\ & \leq  \dismax_{j=1,\ldots,r}e^{\ha_{T,j}}||x-\hx||,\end{array}$$ with $\ha_{T,j}:=\max\{(\m_j+\d)T,0\}$. 
 
Now we can define $S_{T,\e}$: let $\hs_{T,j}:=\frac{1}{\e}e^{\ha_{T,j}}$,  $\hat{\e}_{T,j}=\hs_{T,j}^{-1}$, and create a grid along each coordinate direction: $$ S_{T,\e,i}^j:=\{\a_i^j+\hat{\e}_{T,j},\a_i^j+3\hat{\e}_{T,j}, \ldots,\b_i^j-\hat{\e}_{T,j}\}$$ with cardinality $\ell_i^j\hs_{T,j}/2$.  Define $S_{T,\e}\subset \tilde{K}$ to be the induced lattice $$S_{T,\e}:=S_{T,\e,1}^1\times\ldots \times S_{T,\e,n_1}^1\times\ldots\times S_{T,\e,1}^r\times\ldots\times S_{T,\e,n_r}^r,$$ which by construction is $(T,\e)$-spanning.

Counting points,  $$\#S_{T,\e}=\disp_{i=1}^r\disp_{j=1}^{n_i}\#S_{T,\e,i}^j=\disp_{j=1}^r\disp_{i=1}^{n_j}\ell_i^j\hs_{T,j},$$ and this quantity upper bounds the minimal cardinality $s(T,\e)$ of $(T,\e)$-spanning sets.  Applying the definition of entropy, we have \begin{equation}\label{eq:cardcount} h(A) :=\disl_{\e\rightarrow0}\dislimsup_{T\rightarrow\infty}\frac{1}{T}\log s(T,\e)\leq \disl_{\e\rightarrow0}\dislimsup_{T\rightarrow\infty}\frac{1}{T}\log\#S_{T,\e}.\end{equation} But \begin{equation}\label{eq:finallinear}\begin{array}{ll} \dfrac{1}{T}\#S_{T,\e} & = \dfrac{1}{T}\diss_{j=1}^r\diss_{i=1}^{n_j}\left(\ha_{T,j}+\log(\ell_i^j/2\e)\right)\\ & =\dfrac{1}{T}\left(\diss_{j=1}^rn_j\ha_{T,j}\right)+\dfrac{1}{T}\left(\diss_{j=1}^r\diss_{i=1}^{n_j}\log(\ell_i^j/2\e)\right)\\ & = \diss_{j=1}^rn_j(\m_j+\d)+\dfrac{1}{T}\left(\diss_{j=1}^r\diss_{i=1}^{n_j}\log(\ell_i^j/2\e)\right).\end{array}\end{equation}Taking the limit supremum $T\rightarrow\infty$ on both sides of \eqref{eq:finallinear} we obtain the upper bound $\dislimsup_{T\rightarrow\infty}\dfrac{1}{T}\#S_{T,\e}\leq \diss_{j=1}^rn_j(\m_j+\d)$.  Because $\d>0$ was arbitrary, we have $h(A)\leq \diss_{j=1}^rn_j\m_j=H(A)$, completing one direction.

In the other direction, we consider spanning sets as dependent on the set $K$ of initial conditions. Observe that for any subspace $E\subset \R^n$, the projection $\p_E(S_\e(T))$ of any $(T,\e,K)$-spanning set $S_\e(T)$ still $(T,\e,\p_E(K))$-spans the evolution of $K$ on its projection in the subspace. 

 Because each $\e$-ball in the $||\cdot||_\infty$-norm has finite volume $v_\e$, the volume of the flow $\m(\ph(T,K))$ must be upper bounded by the volumes of each $\e$-ball: $$\mu(\ph(T,K))<\#S_\e(T)v_\e,$$  where here  $\mu$ denotes the Lebesgue volume on $\R^n$. Similarly, for every subspace $E\subset \R^n$, we have $$\m_e\big(\p_E(\ph(T,K))\big)<\#S_\e(T)v_\e^e,$$ where $e=\dim(E)$, $\m_e$ denotes the $e$-dimensional Lebesgue measure and $v_\e^e$ the dimension-$e$ volume of an $\e$-ball.  

Now we choose $E$ to be the unstable subspace,  let $J_1,\ldots,J_\ell$ be the corresponding Jordan blocks with positive real part eigenvalues, and let $$ A^+:= \begin{pmatrix} J_1 & \cdots & 0\\  \vdots & \ddots & \vdots \\ 0 & \cdots & J_k\end{pmatrix}.$$ 

Then: $$\begin{array}{ll} \mu_{e}\big(\pi_E(\varphi(T,K))\big) &= \mu_{e}(e^{A^+T}\pi_E (K)) \\ & = \det(e^{A^+T})\mu_{e}(\pi_E (K)) \\ & =e^{\text{tr}(A^+T)}\mu_{e}(\pi_E (K))\\ & = e^{H(A)T}\mu_{e}(\pi_E (K)).\end{array}$$ The second inequality is simply volume transformation and the third is Liouville's trace formula, which is easily verifiable (c.f.\ \citep[Proposition 15.20]{tu}).  Taking logs and limits, we have: $$\dislimsup_{T\rightarrow\infty}\frac{1}{T}\log\#S_\e(T)\geq\dislimsup_{T\rightarrow\infty} \left(H(A)+\frac{\log\mu_{e}(\pi_E K)/v^e_\e}{T}\right)=H(A)$$ because $K$ is bounded.  As this inequality holds for arbitrary spanning set $S_\e(T)$, in particular it holds for a minimal (with respect to cardinality) spanning set, $$\dislimsup_{T\rightarrow\infty}\frac{1}{T}\log \left(s(T,\e)\right)\geq H(A).$$ Because $\e$ does not appear on the right-hand side, the inequality remains true taking the limit as $\e\rightarrow0$ and thus we obtain that  $h(A)\geq H(A)$, proving the other direction and completing the proof.   \end{proof}

\chapter{Nonscalar Switched Linear Systems}

\section{Introduction} 
Next we would like to turn our attention to a general linear switched system \begin{equation} \label{eq:vectorswitch} \dx = A_\s x\end{equation}  with $A_i\in M_n(\R)$ for $i=1,\ldots, k$, and fixed switching signal $\s:[0,\infty)\rightarrow\{1,\ldots,k\}$.  For a matrix $A$, recall that $tr(A)$ is the sum of all eigenvalues. In general, eigenvalue analysis provides little insight into state expansion, except when there is some Lie algebra structure among the matrices $A_1,\ldots, A_k$. We will momentarily state and prove an upper bound for the entropy of a switched system in which the Lie Algebra generated by $\{A_1,\ldots,A_k\}$ is solvable, in terms of the individual entropies $h(A_i)$.  First, we give a crude lower bound, independent of the switching signal $\s$ and structure on the Lie algebra. 

For simplicity of notation, we will  now once and for all assume $K$ is a unit hypercube, i.e.\,  $\m(K)=1$,  and projected onto any $l$-dimensional subspace spanned by coordinate vectors, the $\ell$-dimensional volume $\m_\ell(\p_{i_1,\ldots,i_{\ell}}(K))=1$.\footnote{Our motivation is that initial volume does not enter into entropy bounds, and we do not want to carry extraneous information (and cluttered notation) through computations.  Of course, this equality depends on first fixing the basis, but as we see from \citep[\S3.1.b]{katoksystems}, entropy is invariant to change-of-basis, and  \textit{a fortiori} to scaling along each coordinate.}

\section{General Lower Bound}
Recall that $\t_i(t)=\disg_0^t\chi_i(s)ds$ where $\chi_i$ is the indicator function which is nonzero only when mode $i$ is active, and that $\tf_i=\dislimsup_{t\rightarrow\infty} \frac{1}{t}\t_i(t)$.

\begin{prop}\label{prop:nolie} Fix $\s:[0,\infty)\rightarrow\{1,\ldots,k\}$ and consider switched linear system \eqref{eq:vectorswitch}. The entropy is lower bounded by $$ h(A_\s)\geq \diss_{i=1}^ktr(A_i)\tf_i.$$\end{prop}

\begin{proof} 
Recall from the lower bound argument in the proof of Theorem \ref{theorem:linear} that there is a uniform bound $v_\e$ on the ratio of volume growth of $K$ along solutions and the number of points in a $(T,\e)$-spanning  set $S_\e(T)$, given by $\mu(\ph(T,K))<v_\e\#S_\e(T),$ and hence $$\#S_\e(T)\geq \frac{\mu(\ph(T,K))}{v_\e}.$$

Now suppose there are $N(t)-1$ switches on the interval $[0,t]$.  Then a general solution of $\dx = A_\sigma x$ is given by $\varphi(t,x)=\disp_{i=1}^{N(t)}e^{A_{j_i}T_i}x$, where $\diss_{i=1}^{N(t)}T_i=t$ and  $\diss_{i=1}^{N(T)} T_i\chi_{\ell}(i)=\tau_\ell(t)$, with $\x_\ell(i)$ indicating whether $j_i=\ell$.  

Then: $$\begin{array}{ll} \mu(\ph(T,K)) & =  \mu\left(\disp_{i=1}^{N(t)}e^{A_{j_i}T_i}K\right) \\ & = \mu\left(e^{A_{j_{N(t)}}T_{N(t)}}(\disp_{i=1}^{N(t)-1}e^{A_{j_i}T_i}K)\right)\\
& = \det(e^{A_{j_{N(t)}}T_{N(t)}})\mu\left(\disp_{i=1}^{N(t)-1}e^{A_{j_i}T_i}K\right)\\
& \;\;\;\;\;\;\;\;\;\;\;\;\;\;\;\;\;\;\vdots\\
& =\left(\disp_{i=1}^{N(t)}\det(e^{A_{j_i}T_i})\right)\mu(K)\\ \end{array}$$ 
$$
\begin{array}{ll}

& = \left(\disp_{i=1}^{N(t)}\det(e^{A_{j_i}T_i})\right)\\
& = \disp_{i=1}^k\det(e^{A_i\tau_i(t)})\\ 
& = \disp_{i=1}^ke^{\text{tr}(A_i)\tau_i(t)}\\
& = e^{\sum_{i=1}^ktr(A_i)\t_i(t)}.\end{array}$$

The individual steps in this volume argument are identical to the analogous computation in Theorem \ref{theorem:linear}, but here we are using the fact that even if matrices do not commute, their determinants do. 

For each $i=1,\ldots,k$,  $\dislimsup_{t\rightarrow\infty} \frac{1}{t}\t_i(t)=\tf_i$, and so $$ \dislimsup_{t\rightarrow\infty}\diss_{i=1}^k\frac{1}{t}tr(A_i)\t_i(t)=\diss_{i=1}^ktr(A_i)\tf_i.$$  Then for each $\d>0$ there is a $T_\d$ and sequence $\Tc=\{t_j\}_{j\in\N}\subset \R$, with $t_j\xrightarrow{j\rightarrow\infty}\infty$ such that $$e^{\sum_{i=1}^ktr(A_i)\t_i(t_j)}\geq e^{\sum_{i=1}^ktr(A_i)(\tf_i-\d)t_j}$$ for  all $t_j\in \Tc\cap \{t>T_\d\}$.  

Then $$\dislimsup_{T\rightarrow\infty}\frac{1}{T}\log\# S_\e(T)\geq \dislimsup_{T\rightarrow\infty}\big(\diss_{i=1}^ktr(A_i)(\tf_i-\d)+\frac{\log(\m(K)/v_\e)}{T}\big)=\diss_{i=1}^ktr(A_i)(\tf_i-\d).$$ As $\d>0$ is arbitrarily, this shows that $h(A_\s)\geq \diss_{i=1}^ktr(A_i)\tf_i$, as claimed.

\end{proof}

\section{Simultaneous Diagonalizability Case}

Eigenvalue analysis is for the most part irrelevant in  computations of entropy for switched (or more generally: time-varying)  linear systems, as the culprit of instability is mixing in the geometry of eigenspaces. Given sufficient structure on the Lie algebra generated by the matrices representing the dynamics, eigenvalue analysis can provide useful information about system instability and expansion. This is true in the case of individually stable systems \citep[Theorem 2.7]{liberzonswitch} for stability analysis, and it is true too in the case of entropy. 

We briefly mention a nonscalar switched linear case in which entropy bounds can be computed in a way similar to the scalar case.  Suppose that the matrices $\{A_1,\ldots,A_k\}$ pairwise commute and are diagonalizable (hence: they are simultaneously diagonalizable \citep[\S1.3]{horn-johnson-book}.   Thus, given matrices $A_1,\ldots,A_k$, there exists some similarity transform $P\in Gl_n(\R)$ which brings $PA_iP^{-1}$ into a matrix in diagonal form, for all $i=1,\ldots,k$. As entropy is invariant to such transformation (again \citep[\S3.1.b]{katoksystems}), we treat the matrices $A_1,\ldots,A_k$ as being already in diagonal (and in the next section, triangular) form. 

Let $(a_i^j)$ denote the $i$-th diagonal element of matrix $A_j$: $$A_j=\begin{pmatrix} a_1^j & \cdots & 0\\ \vdots  & \ddots & \vdots \\ 0 & \cdots & a_n^j\end{pmatrix}. $$ We fix notation which we will use in the following results; define $$\k_i(t):=\diss_{j=1}^ka_i^j\t_j(t),$$ and $$\kf_i:=\diss_{j=1}^ka_i^j\tf_j,$$ where $\tf_j=\dislimsup_{t\rightarrow\infty}\frac{1}{t}\t_j(t)$. We have defined these quantities before in the scalar case, and here we have an analogous scalar definition, treating each state dimension separately. 

\begin{prop}\label{prop:simdiag} Fix switching signal $\s:[0,\infty)\rightarrow\{1,\ldots,k\}$ and consider switched linear system \eqref{eq:vectorswitch} where each $A_1,\ldots,A_k$ pairwise commute and are diagonalizable.  Then the entropy of this system is given by $$\dismax_{i=1,\ldots,k}\kf_i^+\leq h(A_\s)\leq  \diss_{i=1}^k\kf_i^+.$$	
\end{prop}
\begin{proof}   First we show that the expression on the right-hand side is an upper bound by computing separation of solutions starting from different initial conditions, then we construct a $(T,\e)$-spanning set, and count.  To show the lower bound, we use the volume argument. 

Notice that a solution $\ph(t,x)=e^{\int_0^tA_{\s(s)}ds}x$ is given by componentwise scalar solutions in each diagonal entry: $$\ph(t,x) = \begin{pmatrix} e^{\k_1(t)} & \cdots & 0 \\ \vdots & \ddots & \vdots \\ 0 & \cdots & e^{\k_n(t)}\end{pmatrix} \begin{pmatrix} x_1\\ \vdots \\ x_n\end{pmatrix},$$  where $x_i$ denotes the $i$-th entry of the vector $x\in\R^n$.  Separation of solutions on finite time horizon $[0,T]$ is given by \begin{equation}\label{eq:diagswitch}  \begin{array}{ll}||\ph(\cdot,x)-\ph(\cdot ,\hx)||_{[0,T]} & = \displaystyle\sup_{t\in[0,T]}||\ph(t,x)-\ph(t,\hx)||\\ & = \displaystyle\max_{i=1,\ldots,n}e^{\k_i(t)}|x_i-\hx_i|. \end{array}\end{equation} At this point, we combine the results from the scalar switched case in Proposition \ref{prop:scalin} and the linear time-invariant nonscalar case in Theorem \ref{theorem:linear}.  Set $a_{\max}:=\max\{|a_i^j|:\,i=1,\ldots,n,\,j=1,\ldots,k\}$, and by definition of $\k_i$ for any $\d>0$ there is $T_\d>0$ such that $$\frac{1}{t}\k_i(t)<\kf_i+\d$$ whenever $t>T_\d$, for each $i=1,\ldots,n$.  Set $\ha_{T,i}:=\max\{a_{\max}T_\d,(k_i+\d)T\}$ so that $$||\ph(\cdot,x)-\ph(\cdot,\hx)||_{[0,T]} = \dismax_{i=1,\ldots,n}e^{\k_i(t)}|x_i-\hx_i|\leq \dismax_{i=1\ldots,n}e^{\ha_{T,i}}|x_i-\hx_i|.$$ 

 We construct a $(T,\e)$-spanning set as follows.  Define $\hs_{T,i}:=\frac{1}{\e}e^{\ha_{T,i}}$ and $\he_{T,i}:=\hs_{T,i}^{-1}$.  We start with  grid $S_{T,\e,i}$ containing $\hs_{T,i}$ evenly spaced points separated by distance $\he_{T,i}$ and define $$S_{T,\e}=\disp_{i=1}^nS_{T,\e,i},$$ which is by construction $(T,\e)$-spanning.  Then $$\log(\# S_{T,\e}) =\log\left(\disp_{i=1}^n\frac{1}{\e}e^{\ha_{T,i}}\right) = \log(1/\e^n)+\diss_{i=1}^n\ha_{T,i}.$$  As this spanning set provides an upper bound on the cardinality of a minimal $(T,\e)$-spanning set, we compute upper bound for entropy as $$ \begin{array}{ll} h(A_\s) & =\dislim_{\e\rightarrow0}\dislimsup_{T\rightarrow\infty} \frac{1}{T}\log(s(T,\e)) \\ & \leq \dislim_{\e\rightarrow0}\dislimsup_{T\rightarrow\infty}\frac{1}{T}\big(\log(1/\e^n)+\diss_{i=1}^n\ha_{T,i}\big)\\ & = n\d+\diss_{i=1}^n\max\{\kf_i,0\}.\end{array}$$ Indeed, the last equality follows from the fact that whenever $\kf_i\leq 0$, $\ha_{T,i}=a_{\max}T_\d$ and $\disl_{T\rightarrow\infty}\frac{a_{\max}T_\d}{T}=0$.  As $\d>0$ is arbitrary, we obtain the desired upper bound $\diss_{i=1}^n\max\{\kf_i,0\}$.  
 
 The lower bound follows by applying projecting lower bound argument on each component.  Namely, the volume  of $\p_i(\ph(T,K))$---the projection of $\ph(T,K)$ onto the $i$-th component---is bounded above by $2\e \#S_{T,\e,i}$. For each $i=1,\ldots,k$ there is a sequence $\{t_j^i\}_{j\in\N}$, with $t_j^i\xrightarrow{j\rightarrow\infty}\infty$, for which $e^{\k_i(t_j^i)}\geq e^{(\kf_i-\d)t_j^i}$ for all $t_j^i$.   Of course, $\#S_{T,\e}\geq \#S_{T,\e,i}$, so $$\begin{array}{ll}\dislimsup_{T\rightarrow\infty}\frac{1}{T}\log\#S_{T,\e} & \geq \dislimsup_{T\rightarrow\infty}\frac{1}{T}\log\#S_{T,\e,i}\\ & \geq \dislimsup_{j\rightarrow\infty}\big(\frac{1}{t_j^i}\k_i(t_i) -\log(1/2\e)\big)\\ & \geq \dislimsup_{j\rightarrow\infty}(\kf_i-\d)\\ & = \kf_i-\d,\end{array}$$ which proves, since $i$ and $\d$ were arbitrary,  that $\dismax_{i=1,\ldots,k}\kf_i\leq h(A_\s)$, as desired.\end{proof}

We point out a recurring phenomenon  with entropy of switched systems: though given enough structure it is possible to construct bounds using eigenvalues, those bounds depend on how the invariant subspaces of each respective system line up.  Consider the following example. \begin{example}

 $$\begin{array}{ll} \dx &= \begin{pmatrix} 2 & 0 \\  0 & 0\end{pmatrix} \chi_1(\s(t)) x + \begin{pmatrix} 2 & 0 \\ 0 & -1 \end{pmatrix} \chi_2(\s(t))x\\ \dx & = \begin{pmatrix} 2 & 0 \\ 0& 0 \end{pmatrix}\chi_1(\s(t)) x + \begin{pmatrix} 
	-1 & 0 \\ 0 & 2\end{pmatrix}\chi_2(\s(t))x 
\end{array}$$ both of whose systems have individual dynamics with equal entropies (namely 2 and 1, from Theorem \ref{theorem:linear}).  Let $$\s(t):=\left\{\begin{array}{lll}1 & \mbox{when} & \lfloor t\rfloor = 1\,\mbox{mod}\,2\\ 2 & \mbox{else} & \end{array}\right.$$ so that $\tf_1=\tf_2=1/2$. Applying the proposition, the entropy  for the first system $A_\s^1$ is bounded as $$\begin{array}{ll}\max_i\kf_i^+& =\max\{\frac{1}{2}\cdot 2+\frac{1}{2}\cdot 2, \frac{1}{2}\cdot 0+\max\{0,\frac{1}{2}\cdot (-1)\}\} \\ & \leq h(A_\s^1)\\ &  \leq \frac{1}{2}\cdot 2+\frac{1}{2}\cdot 2+ \frac{1}{2}\cdot 0+\max\{0,\frac{1}{2}\cdot (-1)\},\end{array} $$ and so is equality $h(A_\s^1) = 2$.  The entropy for $A_\s^2$, on the other hand, is bounded as $$\max\{1/2,1\} \leq h(A_\s^2) \leq 1/2+1,$$ and in particular $h(A_\s^1)\neq h(A_\s^2)$. 
\end{example}

\section{Simultaneous Triangularizability Case}
We now turn our attention to the switched linear case for which there exists a transformation which simultaneously takes each matrix to one in the form of an upper triangular matrix. Since entropy is invariant to change-of-basis, we treat simultaneous triangularizability as already simultaneously triangularized. In terms of a guiding thread in our investigation of switched systems, this condition is equivalent to a certain property of the Lie algebra generated by these matrices.  At a high level, the Lie algebra encodes information about ``how well'' matrices commute: if  iterated commutators on the generators of the Lie algebra eventually vanish, the Lie algebra is said to be solvable. A set of commuting matrices is solvable because the commutator vanishes at one application.  In general, solvability of the Lie algebra is equivalent to the existence of a simultaneous triangularizing transformation \citep[\S S5.5]{serrelie}. 

To gain insight into why simultaneous triangularizability might be helpful, consider that the dynamics in the last (n-th) entry are totally decoupled from the rest, and each successive component is decoupled from the ones which precede it.  Thus we can consider each system as a scalar linear system and step-by-step solve each system using solutions from one component as input to the dynamics of the components before.  

\begin{prop}\label{prop:solvableswitch} Consider switched system \eqref{eq:vectorswitch} and suppose  that the set of matrices $A_1,\ldots, A_k$ are upper triangular, with $$ A_j= \begin{pmatrix} a_1^j & \cdots & *\\
	\vdots & \ddots & \vdots \\0 & \cdots & a_n^j\end{pmatrix}. $$ Define $N(T):=\#\{t\in[0,T]:\,\disl_{s\nearrow t} \s(s) \neq \disl_{s\searrow t}\s(s)\}$ the number of switching instances on time horizon $[0,T]$, including endpoints $0$ and $T$, and suppose that $\disl_{T\rightarrow\infty} \frac{\log N(T)}{T}=0$.  Let $\kf_i=\diss_{j=1}^k\tf_ja_i^j$ as usual. Then entropy is bounded above by 
	$$ h(A_\s) \leq  \diss_{i=1}^{n}(\kf_1^++\kf_2^++\ldots +\kf_i^+)= n\kf_1^++(n-1)\kf_2^++\ldots+\kf_n^+.$$
\end{prop}

The mild assumption on $N(T)$ guarantees that the switching rate is subexponential. It is certainly weaker than average dwell time \citep[\S3.2]{liberzonswitch}, and permits increasingly frequent switches, with a generous upper bound on the rate of frequency increase.

\begin{lemma}\label{lemma:bound} Let $\d>0$ be arbitrary.  Then for each $i=1,\ldots,n$, there is $T_\d>0$ and constant $c>0$ such that $$\disg_0^te^{\k_i(s)}ds\leq cN(T)e^{(\kf_i^++\d)t}$$ whenever $T>T_\d$, where $N(T)$ denotes as defined above the number of switches on $[0,T]$. In particular, when $\kf^+_i<0$, the upper bound is simply $cN(T)e^{\d t}$. \end{lemma} 
\begin{proof} First fix $t>0$, and partition $[0,t]$ as $\Pc_t:=\{t_1,\ldots,t_N\}$, the set of switching times with $t_1=0$ and $t_N=t$, with subscript $N=N(t)$.  Finally, let $a_i^{\s(t_\ell)}$ denote the system active on $(t_\ell,t_{\ell+1})$.  We compute the integral by breaking it up at switching points: $$\begin{array}{ll}\disg_0^te^{\k_i(s)}ds & = \diss_{\ell=1}^{N-1}\disg_{t_\ell}^{t_{\ell+1}}e^{\k_i(s)}ds\\
& = \diss_{\ell=1}^{N-1}e^{\k_i(t_\ell)}\disg_{t_\ell}^{t_{\ell+1}}e^{a_i^{\s(t_\ell)}(s-t_l)}ds \\
& =\diss_{\ell=1}^{N-1}e^{\k_i(t_\ell)}\frac{1}{a_i^{\s(t_\ell)}}(e^{a_i^{\s(t_\ell)}(t_{\ell+1}-t_\ell)}-1)\\
& = \diss_{\ell=1}^{N-1}\frac{1}{a_i^{\s(t_\ell)}}(e^{\k_i(t_{\ell+1})}-e^{\k_i(t_\ell)}).\end{array} $$ Define $\frac{1}{a_i^{\s(t_{0})}}:=0$, $\frac{1}{a_i^{\s(t_{N})}}:=0$, and $c:=\max\left\{\left|\frac{1}{a^\a}\right|+\left|\frac{1}{a^\b}\right|:\,1\leq \a,\b\leq k\right\}$ so that $$ \begin{array}{ll} \diss_{\ell=1}^{N-1}\frac{1}{a_i^{\s(t_\ell)}}(e^{\k_i(t_{\ell+1})}-e^{\k_i(t_\ell)}) & =\diss_{\ell=1}^{N}e^{\k_i(t_\ell)}(\frac{1}{a_i^{\s(t_{\ell-1})}}-\frac{1}{a_i^{\s(t_{\ell})}})\\ & \leq c\diss_{\ell=1}^{N}e^{\k_i(t_\ell)}.\end{array}$$

To bound the right-hand side, let $\d>0$ be arbitrary.  We choose $T_\d$ such that $\frac{1}{t}\k_i(t)<\kf_i^++\d$ whenever $t>T_\d$.  On the interval $[0,t]$, and on any closed subinterval thereof, $k_i(s)$ is bounded and takes maximum value.  Because $\kf_i^+\geq 0$, there is some $T_{T_\d}>0$ such that $ (\kf_i^++\d)t\geq \dismax_{s\in[0,T_\d]}\k(s)$ whenever $t>T_{T_\d}$.  Then for $T:=\max\{T_\d,T_{T_\d}\}$,  both $\dismax_{s\in[0,T_\d]}\k_i(s)\leq (\kf^++\d)t$ and $\k_i(t)<(\kf^++\d)t$ for all  $t>T$, and for such $t$, we have $$c\diss_{\ell=1}^Ne^{\k_i(t_\ell)}\leq cN(t)e^{(\kf_i^++\d)t},$$ as desired.\end{proof}

There is one more observation before beginning the proof. Because entropy looks at behavior ``in the limit'' we generally do not care what happens on some fixed initial segment.  In proof of the proposition, we pretend like a set which works for all time in the tail, works for all time, simply.  The lemma makes this remark more precise. 
\begin{lemma} Fix $\e>0$ and suppose that for every $\d>0$ there is some $T_\d>0$ such that for all $T>T_\d$ and any given set $S_\d(T,\e)$ which is $(T,\e)$-spanning on the restriction to $[T_\d,T]$---i.e., for all $x\in K$, there is some $\hx \in S_\d(T,\e)$ such that $||\ph(\cdot,x)-\ph(\cdot,\hx)||_{[T_\d,T]}$---then there is a discrete set of points $S(T_d)$ making $S(T,\e):=S(T_d)\cup S_\d(T,\e)$ a $(T,\e)$-spanning set.  Moreover, $$\dislimsup_{T\rightarrow\infty}\frac{1}{T}\log\#S(T,e)=\dislimsup_{T\rightarrow\infty}\frac{1}{T}\#S_\d(T,\e).$$\end{lemma} 
\begin{proof} By definition, $$||\ph(\cdot,x)-\ph(\cdot,\hx)||_{[0,T]} = \dismax\{||\ph(\cdot,x)-\ph(\cdot,\hx)||_{[0,T_\d]}, ||\ph(\cdot,x)-\ph(\cdot,\hx)||_{[T_\d,T]}\}.$$ The first claim is equivalent to the existence of a finite $(T,\e)$-spanning set, which follows directly from compactness of $\ph(K,T)$ w.r.t.\ the $||\cdot||_{[0,\cdot]}$ norms. The second follows from this: as $\#S(T_\d)<\infty$ is independent of $T$, we have $$\begin{array}{ll}\dislimsup_{T\rightarrow\infty}\frac{1}{T}\log\#S(T,e) & =\dislimsup_{T\rightarrow\infty}\frac{1}{T}\log\#(S_\d(T,e)\cup S(T_\d))\\ & \leq \dislimsup_{T\rightarrow\infty}\frac{1}{T}(\log\#S_\d(T,\e)+\log\#S(T_\d))\\ & = \dislimsup_{T\rightarrow\infty}\frac{1}{T}\log\#S_\d(T,\e),\end{array}$$ which is what we wanted to show.\end{proof}

\begin{proof}[Proof of Proposition]
	
	We fix notation:   we let $x_i\in \R$ denote the $i$-th entry of $x\in\R^n$ and we use $\ph_i$ to denote the $i$-th coordinate of solution $\ph(t,x)$ for the system. Furthermore, for matrix $A_j\in M_n(\R)$, we notate its entries as $$(A_j)_{\a,\b}=\left\{\begin{array}{lll}a_\a^j & \mbox{if} & \a=\b\\ b_{\a,\b}^j & \mbox{else,} & \end{array}\right.$$ where by assumption $b_{\a,\b}=0$ whenever $\a>\b$. 	 In this argument, we will apply variation of  constants with solutions $\ph_{i+1},\ldots,\ph_n$ as inputs in system $$\dx_i=f(x_i,\ldots,x_n)=a_i^\s x_i+\diss_{\b=i+1}^nb_{i,\b}^\s x_\b$$ in order to bound separation of solutions, and thereby construct $(T,\e)$-spanning sets.  By definition of the $\infty$-norm, on finite time horizon $[0,T]$ we have $$||\ph(\cdot,x)-\ph(\cdot,\hx)||_{[0,T]}= \dissup_{t\in [0,T]}||\ph(t,x)-\ph(t,\hx)|| = \dissup_{t\in[0,T]}\dismax_{i=1,\ldots,n}|\ph_i(t,x)-\ph_i(t,\hx)|.$$

As this latter expression is equal to $\dismax_{i=1,\ldots,n}\dissup_{t\in[0,T]}|\ph_i(t,x)-\ph_i(t,\hx)|$, we compute a bound on each  $\ph_i$, $i=1,\ldots,n$, starting from $i=n$.  Because $\dx_n= a_n^\s x_n$ is a scalar linear switched system, its solution is given by $$ \ph_n(t,x) = e^{\sum_{j=1}^ka_n^j\t_j(t)}x_n=e^{\k_n(t)}x_n.$$  As $\dislimsup_{t\rightarrow\infty} \frac{1}{t}\k_n(t)=\kf_n$, for $\d>0$ there is $T_\d'$ such that $\k_n(t)<(\kf_n+\d)t\leq (\kf_n^++\d)t$ for all $t>T_\d'$.  As $\k_n(s)$ is continuous, it is bounded on $[0,T_{\d}']$ and since $\kf_n^++\d>0$, there is $T_\d''$ such that $\dissup_{s\in[0,T_\d']}\k_n(s)<(\kf_n^++\d)t$ whenever $t>T_\d''$.  So for $T_\d:=\max\{T_\d',T_d''\}$, we have $\dissup_{s\in[0,T]}\k_n(s)<(\kf_n^++\d)T$ for all $T>T_\d$. Therefore, for such $T>T_\d$ and $t\in[0,T]$,  \begin{equation}|\ph_n(t,x)|\leq e^{\k_n(t)}|x_n|\leq e^{(\kf^+_n+\d)T}|x_n|= p(N(T))e^{(\kf_n^++q(\d)\d)T}|x_n|,\end{equation} where we write $p(x)=q(x)=1\in\R[x]$ only to anticipate the induction step.  

Applying the argument above with $T_\d'$ and $T_\d''$, for each $\k_i$, we have for some $T_\d(i)$ that $\dissup_{s\in[0,T]}\k_i(s)<(\kf_i^++\d)T$ whenever $T>T_\d(i)$, and from here onward we will let $T_\d=\max\{T_\d(1),\ldots,T_\d(n)\}$. 
	
	We argue recursively, $(n-1)\rightarrow 1$, and we first start with an explicit computation of the second case when $i=n-1$. For $\dx_{n-1} = a_{n-1}^\s x_{n-1}+b_{n-1,n}^\s x_n$, by variation of constants, $$\begin{array}{ll} \ph_{n-1}(t,x) & = e^{\k_{n-1}(t)}(x_{n-1}+\disg_0^tb_{n-1,n}^{\s(s)}\ph_n(s,x)e^{-\k_{n-1}(s)}ds).\end{array}$$  Notice that we factor out $\k_{n-1}(t)$ from the integral since $\k_{n-1}(t-s)=\k_{n-1}(t)-\k_{n-1}(s)$, which follows directly from the definition of $\k_{n-1}(t)$. 
	
	 Let $b:=\max\{|b_{\a,\b}^j:\,1\leq j\leq k,\,1\leq \a,\b\leq n\}$ and first assuming that $\kf_{n-1}^+>0$ bound the solution as: 
	
\begin{equation}\label{eq:firstbound} \begin{array}{ll} |\ph_{n-1}(t,x)| & =
\left|e^{\k_{n-1}(t)}\big(x_{n-1}+\disg_0^tb_{n-1,n}^{\s(s)}\ph_n(s,x)e^{-\k_{n-1}(s)}ds\big)\right| \\ 
& \leq e^{\k_{n-1}(t)}\big(|x_{n-1}|+b\disg_0^t|\ph_n(s,x)|e^{-\k_{n-1}(s)}ds\big) \\
& \leq e^{(\kf_{n-1}^++\d)T}\big(|x_{n-1}|+b\disg_0^t|\ph_n(s,x)|e^{-\k_{n-1}(s)}ds\big)\\
& \leq e^{(\kf_{n-1}^++\d)T}\big(|x_{n-1}|+b\disg_0^te^{(\kf_n^++\d)T}|x_n|e^{-\k_{n-1}(s)}ds\big)\\
& =e^{(\kf_{n-1}^++\d)T}\big(|x_{n-1}|+be^{(\kf_n^++\d)T}\disg_0^te^{-\k_{n-1}(s)}ds|x_n|\big).\\
\end{array}\end{equation}

So far we only applied the bound from the previous computation on $\ph_n$ and the exponential bound from $\dislimsup_{t\rightarrow\infty} \frac{1}{t}\k_{n-1}(t)=\kf_{n-1}$.  

Continuing with the computation, now applying the particular case with $\kf_i^+<0$ of  Lemma \ref{lemma:bound}, we obtain:

\begin{equation}\label{eq:preinduction}\begin{array}{ll}|\ph_{n-1}(t,x)| & \leq e^{(\kf_{n-1}^++\d)T}\big(|x_{n-1}|+be^{(\kf_n^++\d)T}cp(N(T))e^{\d T}|x_n| \big)\\
& \leq e^{(\kf_{n-1}^++\d)T}\big(|x_{n-1}|+p'(N(T))e^{(\kf_n^++\d)T}e^{\d T}|x_n| \big)\\
& \leq p''(N(T))\big(e^{(\kf_{n-1}^++q_{n-1}(\d)\d)T}|x_{n-1}|+e^{(\kf_n^++\kf_{n-1}^++q_n(\d)\d)T}|x_n|\big)\\
& \leq p''(N(T))\big(e^{(\kf_{n-1}^++q(\d)\d)T}|x_{n-1}|+e^{(\kf_n^++\kf_{n-1}^++q(\d)\d)T}|x_n|\big).\end{array}\end{equation}

Here and in subsequent computations, we use primes to distinguish polynomials which satisfy the inequalities or equalities.  They are not derivatives.  In the preceding sequence, $p'\geq bcp$, $p''\geq \max\{1,p'\}$, $q_{n-1}\geq 1$,  $q_n\geq 3$, and $q\geq \max\{q_{n-1},q_n\}$. 

If, on the other hand, $\kf_{n-1}^+<0$, then we modify the transition from the second to third line in \eqref{eq:firstbound} and following as: $$\begin{array}{l} 
  e^{\k_{n-1}(t)}\big(|x_{n-1}|+b\disg_0^t|\ph_n(s,x)|e^{-\k_{n-1}(s)}ds\big) \\ \leq  e^{(\kf_{n-1}^++\d)t}|x_n-1|+e^{\k_{n-1}(t)}b\disg_0^t|\ph_n(s,x)|e^{-\k_{n-1}(s)}ds\\
 \leq   e^{(\kf_{n-1}^++\d)T}|x_{n-1}|+e^{\k_{n-1}(t)}b\disg_0^te^{(\kf_n^++\d)T}|x_n|e^{-\k_{n-1}(s)}ds\\
\leq   e^{(\kf_{n-1}^++\d)T}|x_{n-1}|+e^{\k_{n-1}(t)}b\disg_0^te^{-\k_{n-1}(s)}dse^{(\kf_n^++\d)T}|x_n|
\end{array}$$ and as $||e^{\k_{n-1}(\cdot)}||_{\Lc_\infty}<\infty$, the term $e^{\k_{n-1}(t)}b\disg_0^te^{-\k_{n-1}(s)}ds\leq bc\disg_0^te^{-\k_{n-1}(s)}ds$, and we apply the lemma again to obtain an expression as in \eqref{eq:preinduction}.

Now for the induction argument; let $T>T_\d$ and assume for $j<i$ that \begin{equation}\label{eq:inductioninequality} ||\ph_{n-j}(\cdot,x)||_{[0,T]}\leq p(N(T))\diss_{\ell=0}^je^{(\kf_{n-j}^++\ldots+\kf_{n-j+\ell}^++q(\d)\d)T}|x_{n-j+\ell}|,\end{equation}  and we want to show that $$ ||\ph_{n-i}(\cdot,x)||_{[0,T]}\leq p'(N(T))\diss_{\ell=0}^ie^{(\kf_{n-i}^++\ldots+\kf_{n-i+\ell}^++q'(\d)\d)T}|x_{n-i+\ell}|$$ for polynomials $p$, $p'$, $q$, and $q'$.

Indeed, $|\ph_{n-i}(\cdot,x)||_{[0,t]}$ is bounded above by $$\begin{array}{l}   e^{\k_{n-i}(t)}\big(|x_{n-i}|+\diss_{\ell=1}^i\disg_0^tb|\ph_{n-i+\ell}(s,x)|e^{-\k_{n-i}(s)}ds\big)\\ 
\leq e^{\k_{n-i}(t)}\big(|x_{n-i}\\ \;\;\;\;\;\;\;\;+\diss_{\ell}^ib\disg_0^t\diss_{\ell'=0}^{i-\ell}p_{\ell,\ell'}(N(T))e^{(\kf_{n-i+\ell}^++\ldots+\kf_{n-i+\ell+\ell'}^++\q_{\ell,\ell'}(\d)\d)T}|x_{n-i+\ell+\ell'}|e^{-\k_{n-i}(s)}ds\big),
\end{array}$$ where we apply induction to each $|\ph_{n-i+\ell}|$ inside the integrand, and we can wlog take polynomials $p(N(T))>\dismax_{\ell,\ell'}p_{\ell,\ell'}(N(T))$ and $q(\d)>\dismax_{\ell,\ell'}q_{\ell,\ell'}(\d)$.

Then this last expression is bounded above by $$\begin{array}{ll} &  e^{\k_{n-i}(t)}\big(|x_{n-i}|+p(N(T))\diss_{\ell=1}^i\diss_{\ell'=0}^{i-\ell}\disg_0^te^{(\kf_{n-i+\ell}^++\ldots+\kf_{n-i+\ell+\ell'}^++q(\d)\d)T}|x_{n-i+\ell+\ell'}|e^{-\k_{n-i}(s)}ds\big)\\
 =&  e^{\k_{n-i}(t)}\big(|x_{n-i}|+p(N(T))\diss_{\ell=1}^i\big(\diss_{\ell=0}^{\ell-1}e^{(\kf_{n-i+\ell}^+\ldots+\kf_{n-i+\ell-\ell'}^++q(\d)\d)T}\big)|x_{n-i+\ell}|\disg_0^te^{-\k_{n-i}(s)}ds\big)\\
 \leq&  e^{\k_{n-i}(t)}|x_{n-i}|+p'(N(T))\diss_{\ell=1}^ie^{\kf_{n-i+\ell}^++\ldots+\kf_{n-i+1}^++q(\d)\d)T}|x_{n-i+\ell}|e^{\k_{n-i}(t)}\disg_0^te^{-\k_{n-i}(s)}ds,
\end{array}$$ and we are left with evaluating $e^{\k_{n-i}(t)}\disg_0^te^{-\k_{n-i}(s)}ds$.  When $\kf_{n-i}>0$, we upper bound this by $p(N(T))e^{(\kf_{n-i}+q(\d)\d)T}=p(N(T))e^{(\kf_{n-i}^++q(\d)\d)T}$ as in the proof of the preceding proposition.  Otherwise, $\kf_{n-i}<0$ and $\dissup_{s\in[0,\infty)}e^{\k_{n-i}(s)}<\infty$, so by Lemma \ref{lemma:bound}, we upper bound this also by $$\big(\dissup_{s\in[0,\infty)}e^{\k_{n-i}(s)}\big)p(N(T))e^{(\kf^+_{n-i}+q(\d)\d)T}\leq p'(N(T))e^{(\kf_{n-i}^++q(\d)\d)T}.$$

Then in particular, $$|\ph_1(t,x)|\leq p(N(T))\big(e^{(\kf_1+q(\d)\d)T}|x_1|+\ldots+e^{(\kf_1+\ldots+\kf_n+q(\d)\d)T}|x_n|\big).$$
		
	Thus our bound on separation of solutions is given by: $$\begin{array}{ll}||\ph(\cdot,x)-\ph(\cdot,\hx)||_{[0,T]}& \leq \dismax_{i=1,\ldots,n}|\ph_i(t,x)-\ph_i(t,\hx)|\\ & \leq p(N(T))(e^{(\kf^+_1+q(\d)\d)T}|x_1-\hx_1|+\ldots+e^{(\kf^+_1+\ldots+\kf^+_n+q(\d)\d)T}|x_n-\hx_n|).\end{array}$$  In order for a canonically constructed grid of equally spaced points along coordinates to be $(T,\e)$-spanning, a sufficient condition is given by $$|x_i-\hx_i|<\frac{2\e e^{-(\kf^+_1+\ldots+\kf^+_i+q(\d)\d)T}}{np(N(T))}=\frac{1}{p'(N(T),\e)}e^{-(\kf^+_1+\ldots+\kf^+_i+q(\d)\d)T}$$ for each $i=1,\ldots,n$.  On an interval of length $1=\m_1(\p_i(K))$, a grid of $$\big(\frac{1}{p(N(T))}e^{-(\kf^+_1+\ldots+\kf^+_i+q(\d)\d)T}\big)^{-1}=p(N(T))e^{(\kf^+_1+\ldots+\kf^+_i+q(\d)\d)T}$$ points will suffice, or in other words $$\#s(T,\e)\leq p(N(T))\disp_{i=1}^ne^{(\kf^+_1+\ldots+\kf^+_i+q(\d)\d)T}.$$
	
	Taking log, dividing by $T$, applying limit supremum in $T$ then limit in $\e$, we obtain \small $$\begin{array}{ll} h(A_\s)&  \leq \disl_{\e\rightarrow0}\dislimsup_{T\rightarrow\infty} \frac{1}{T}\log P(N)\disp_{i=1}^ne^{(\kf^+_1+\ldots+\kf^+_i+q(\d)\d)T}\\ & =\dislimsup_{T\rightarrow\infty} \frac{1}{T}(\log P(N))+\diss_{i=1}^n(\kf^+_1+\ldots+\kf^+_i+nq(\d)\d),\end{array}$$ which proves that $h(A_\s)\leq \diss_{i=1}^n\kf^+_1+\ldots+k_i+nq(\d)\d$.  Now with each inequality in the preceding argument,  we modified the polynomial $q(\cdot)$ in the exponent at most once, but there were only finitely many steps, so we end up with a final $\tilde{q}\in\R[x]$ taking $\d$ as an argument for which $\displaystyle\inf_{\d>0}\tilde{q}(\d)\d=0$.  As $\d>0$ was arbitrary, we therefore conclude that $$h(A_\s)\leq \disinf_{\d>0}\diss_{i=1}^n\kf^+_1+\ldots+k_i+n\tilde{q}(\d)\d=\diss_{i=1}^n(\kf_1^++\ldots+\kf_i^+),$$ completing the proof.\end{proof}

\end{document}